\documentclass{article}

\usepackage{amsmath,amsthm} %
\usepackage{amsfonts, latexsym, amssymb} 
\usepackage{color,pdflscape}
\usepackage{graphicx}
\usepackage{fullpage}

\usepackage{pdflscape}
\usepackage{placeins}

\usepackage{etex,etoolbox}
\usepackage{thmtools}
\usepackage{environ}
\usepackage{tabularx}

\DeclareGraphicsExtensions{.pdf,.png,.gif,.jpg}

\theoremstyle{plain}
\newtheorem{theorem}{Theorem}

\newtheorem{proposition}[theorem]{Proposition}

\theoremstyle{definition}
\newtheorem{assumption}{Assumption}

\theoremstyle{remark}

\newcommand{\coloneq}{\stackrel{\textup{\tiny def}}{=}}

\newcommand{\iid}{i.\@i.\@d.\@ }
\newcommand{\eg}{\emph{e.g.}, }
\newcommand{\ie}{\emph{i.e.}, }

\providecommand{\abs}[1]{\left\lvert#1\right\rvert}

\newcommand{\E}{{\mathbb E}}
\newcommand{\R}{{\mathbb R}}

\newcommand{\neww}[1]{#1}

\DeclareMathOperator{\avg}{avg}
\DeclareMathOperator{\proj}{proj}
\DeclareMathOperator{\prob}{prob}

\renewcommand{\leq}{\leqslant}
\renewcommand{\geq}{\geqslant}

\usepackage[parfill]{parskip} 

\usepackage{url} 
\usepackage{booktabs} 
\usepackage{array} 
\usepackage{paralist} 
\usepackage{verbatim} 

\usepackage[caption=false,font=footnotesize]{subfig}

\begin{document}

\title{Distributionally Robust Optimisation \\ in Congestion Control}

\author{
Jakub Mare{\v c}ek$^{1}$\thanks{{\tt jakub.marecek@ie.ibm.com}}, Robert Shorten$^{2}$, Jia Yuan Yu$^{3}$\\[3mm]
$^{1}$ IBM Research, Ireland \\
$^{2}$ University College Dublin, Ireland \\
$^{3}$ Concordia University, Canada
}

\maketitle

\abstract{
 \neww{The effects of real-time provision of travel-time
 information on the behaviour of drivers are considered.
The model of Marecek et al. 
[Int. J. Control 88(10), 2015]
is extended to consider uncertainty in the response of a driver to an interval
provided per route.
Specifically, it is suggested that one can optimise 
over all distributions of a random variable associated with the driver's response
with the first two moments fixed, and for each route, over the sub-intervals 
 within the minimum and maximum in a certain number of previous realisations of the travel time per the route. 
} 
}

\section{Introduction}

Congestion on the roads is often due to drivers using them in a
synchronized manner, ``a wrong road at a wrong time''.
Intuitively, the synchronisation is partly due to the reliance on 
the same unequivocal information about past traffic conditions,
 which the drivers mistake for a reliable forecast of future traffic conditions. 
Perhaps, if the information about past traffic conditions were provided
 in a different form, the synchronisation could be reduced.
This intuition led to a considerable interest in 
  advanced traveller information systems and models of dynamics of information provision
\cite{arnott1991does,BENAKIVA1991251,bonsall1992influence,arnott1993structural,emmerink1996information,bottom2000consistent,papageorgiou2007its,marecek2015signaling,marecek2016signaling}.
In this paper, we propose and study novel means of information provision.

With the increasing availability of satellite-positioning traces of individual cars,
 it is becoming increasingly clear that there are many approaches to 
 aggregating the information and providing them to the public,
 while it remains unclear what approach is the best. 
Following \cite{marecek2015signaling,marecek2016signaling},
 we model the relationship of information provision and road use as a partially known non-linear dynamical system.
In practice, our approach relies on a road network operator with
  up-to-date knowledge of congestion across the road network,
who broadcasts travel-time information to drivers,
  which is chosen so as to alleviate congestion,
  based on an estimate of the driver's response function, 
  e.g., up to the first two moments of some random variables involved.
\neww{In terms of theory, we study non-linear dynamics,
  which are not perfectly known.}
This poses a considerable methodological challenge.

We make first steps towards modelling the interactions among the road network operator and the drivers over time
  as a stochastic control problem and the related delay-tolerant and risk-averse means of information provision.
In an earlier paper \cite{marecek2015signaling}, we have
  studied the communication of a scalar per route at each time, specific to each driver.
In another recent paper \cite{marecek2016signaling}, 
  we have studied the communication of two scalars (an interval) per route (or road segment) at each time, 
  with the same information broadcast to all drivers.
There, the intervals were based on the minimum and maximum travel time over the segment
  within a time window.
In this paper, we propose an optimisation procedure, where one considers sub-intervals of
  the interval.
Across all three papers, we show that congestion \neww{can be} reduced by withholding some information,
 while ensuring that the information remains consistent with the true past
observations.

Let us consider the travel time over a route as a time series.
Broadcasting the most recent travel time,
 an average over a time window,
 or any other scalar function over a time window,
 may lead to a suboptimal ``cyclical outcome,'' 
 where drivers overwhelmingly pick the supposedly fastest route, 
 leading to congestion therein, and another route being announced as the fastest, \neww{only to become congested in turn}.
On the other hand, depriving the drivers of any information leads to a suboptimal outcome, 
  where each \neww{driver} acts more or less \neww{randomly}.  
We illustrate our findings on an intentionally simple model. 

\begin{table}[t]
\caption{An overview of the related work (top), our suggestion (middle), and suggestions for future work (bottom) within two-parameter route choice formulations and behaviour of the related models.}
\label{tab:int-signals}
    \centering
    \begin{tabularx}{\textwidth}{ X  p{2.6cm} p{5.1cm} p{6.18cm} }
\toprule
Ref.                         & Name                          & $\underline{u}^m_t$    & $\overline{u}^m_t$ \\ 
\midrule
\cite{marecek2015signaling}  & $(\delta, \gamma)$ & \small $c_m(n^m_{t-1}) + \nu_t^m - \delta^m /2$ & \small $c_m(n^m_{t-1}) + \nu_t + \delta^m /2$ \\ 
\cite{marecek2016signaling}  & $r$-extreme        & \small $\arg \min_{j=t-r,\ldots,t-1} \{ c_m(n_j^m) \}$   &  \small $\arg \max_{j=t-r,\ldots,t-1} \{ c_m(n_j^m) \} $ \\ 
\cite{Epperlein2016}                         & smoothing         & \small $q_1 \underline{u}^m_{t-1} + (1-q_1)c_m(n_{t-1}^m)$ & \small $q_2 \overline{u}_{t-1}^m + (1-q_2)\,\bigl|c_m(n_{t-1}^m)- \underline{u}^m_{t-1} \bigl|$       \\ 
\cite{Nikolova2014}  & mean and STD            & \small $\frac{1}{r} \sum_{j=t-r,\ldots,t-1} c_m(n_j^m)$   &  \small $\frac{1}{r} \sum_{j=t-r,\ldots,t-1} (c_m(n_j^m) - \underline{u}^m_{t})^2$  \\[3mm] 
\midrule
 & $r$-supported      & \small $\proj_{\underline{u}^m_t} \arg \min_{ (\underline{u}^m_{t}, \overline{u}^m_{t}) \in P(S_t, \Omega)} C(n_t)$ & \small $\proj_{\overline{u}^m_t} \arg \min_{ (\underline{u}^m_{t}, \overline{u}^m_{t}) \in P(S_t, \Omega)} C(n_t)$ \\[3mm] 
\midrule
  & mean, VaR        & \small $\frac{1}{r} \sum_{j=t-r,\ldots,t-1} c_m(n_j^m)$   &  \small VaR$_\alpha \coloneq \inf\{l \in \mathbb{R}: \prob(L>l)\le 1-\alpha\}$  \\ 
  & mean, CVaR            & \small $\frac{1}{r} \sum_{j=t-r,\ldots,t-1} c_m(n_j^m)$   &  \small CVaR$_\alpha \coloneq \frac{1}{\alpha}\int_0^{\alpha} \mbox{VaR}_{\gamma} d\gamma $ \\ 
\bottomrule
    \end{tabularx}\\[2mm]
\end{table}  

\section{Related work}\label{sec:mod}

Recent studies \cite{marecek2015signaling,marecek2016signaling,Epperlein2016} have focussed on 
a dynamic discrete-time model of congestion, where a
finite population of $N$ drivers is confronted with $M$ alternative routes
at every time step. 
The time horizon is discretized into discrete periods $t=1,2,\ldots$.  
At each time, each driver picks exactly one route, and is hence ``atomic''.
Let $a_t^i$ denote the choice of driver $i$ at time $t$ and $n^m_t = \sum_i 1_{[a_t^i = m]}$ be the
number of drivers choosing route $1 \le m \le M$ at time $t$.
Sometimes, we use $n_t$ to denote the vector of $n^m_t$ for $1 \le m \le M$.
The travel time $c_m(n^m_t)$ of route $m$ at time $t$ is a function of the
number $n^m_t$ of drivers that pick $m$ at time $t$, 
$c_m: \mathbb N \to \mathbb R_+$.
The social cost $C(n_t)$ weights the travel times of the routes at
time $t$ with the proportions of drivers taking the routes, \ie
\begin{align}
  \label{eqn:socialtravel time}
  C(n_t) \triangleq \sum_{m = 1}^{M} \frac{n^m_t}{N} \cdot c_m(n^m_t).
\end{align}
Notice that in the case of two alternatives, $M = 2$, $C(n_t)$ becomes a function of $n^1_t$ only, 
with $n^2_t$ beign equal to $N - n^1_t$: 
\begin{align}
  C(n_t) = \frac{n^1_t}{N} \cdot c_1(n^1_t) + \frac{N - n^1_t}{N} \cdot c_2(N - n^1_t).
\end{align}
The social or system optimum at every
time step $t$ is $n^* \in \arg\min_{0 \le n \le N} C(n)$. 

Notice that the travel time is, in effect, a time-series,
 with a data point per passing driver. 
Often, however, one \neww{may} want to aggregate the time series,
for instance in order to communicate travel times succinctly.
Essentially, \cite{marecek2015signaling,marecek2016signaling,Epperlein2016} discuss various means of aggregating the history of travel times $c_m(n^m_{t'})$ for all $1 \le m \le M$ and for all times $t' < t$ in past relative to present $t$. 
Every driver $i$ takes route $a_t^i$
 based on the history of $s_{t'}, t' \le t$ received up to time $t$.
In keeping with control-theoretic literature, 
 a mapping of such a history to a route is called a \emph{policy}.
$\Omega$ denotes the set of all possible types of drivers and
$\mu$ a probability measure over the set $\Omega$, which describes the
distribution of the population of drivers into types.
We refer to \cite{marecek2016signaling,Epperlein2016} for the measure-theoretic definitions.

\neww{Sending of the most recent travel time or any other single scalar value per route uniformly to all drivers is not socially optimal \cite{marecek2016signaling}.
One option for addressing this issue is to vary the scalar value sent to each user.} 
\cite{marecek2015signaling} studied a scheme,  where the network operator sends a distinct $s^i_t \triangleq (y^{m,i}_t, 1 \le m \le M) \in \mathbb R^{M}$ \neww{to each driver $i$ at time $t$}, where
\begin{align}
  y^{m,i}_t \triangleq c_m(n^m_{t-1}) + w^{i,m}_t, 
\end{align}
and the sequence of random noise vectors $\{ w^{i,m}_t : t=1,2,\ldots\}$ is \iid such that for all $t$, $\E w^{i,m}_t = 0$,
and $w^{i,m}_t - w^{i,{m'}}_t$ is normally distributed with mean $0$ and
variance $\sigma^2$ for $1 \le m \not = m' \le M$.
\neww{These properties of $w^{i,m}_t$ assure that no driver is being disadvantaged over the long run,
but the absolute value of $w^{i,m}_t$ may vary across drivers $i$ at a particular time $t$.}


Considering the introduction of \neww{such driver-specific randomisation} may not be desirable, \cite{marecek2015signaling} presented a scheme that broadcasts \emph{two} \neww{distict} scalar values \neww{per} route to all drivers, \neww{where the two distinct scalars for a particular route are the same for all the drivers at a particular time}.
\neww{For $M$ routes, one has $s_t \triangleq (\underline{u}^m_t,
\overline{u}^m_t, 1 \le m \le M) \in \in \mathbb R^{2M}$, where}
\neww{
\begin{align}
  \underline{u}^m_t &\triangleq c_m(n^m_{t-1}) + \nu^m_t - \delta^m/2,\\
  \overline{u}^m_t &\triangleq c_m(n^m_{t-1}) + \nu^m_t + \delta^m/2,
  \quad m \in \{A,B\},
\end{align}
}
where $\nu_t^m$ are \iid uniform random variables with support:
\neww{
\begin{align*}
  \mbox{Supp}(\nu^m_t) &= [-\delta^m/2,\delta^m/2].
\end{align*}
}
\neww{Notice that \cite{marecek2015signaling}} use $\delta$ and $\gamma$
to denote the non-negative constants $\delta^A$ and $\delta^B$ 
\neww{in the case of $M=2$, and hence use $(\delta,\gamma)$-interval to denote such $s_t$.} 
Let $\Omega$ be a finite subset of $ = [0,1]$ and assume that 
each driver $1 \le i \le N$ is of type $\omega \in
\Omega$ and follows the policy $\pi^\omega$:
\begin{align}
  a_t^i \triangleq \pi^\omega(s_t) \triangleq \arg\min_{m=1}^{M} \; \omega \underline{u}^m_t + (1-\omega) \overline{u}^m_t.
\label{policylambda}
\end{align}
in response to $s_t$.
Observe that for $\omega = 0$, policy $\pi^0$ models a risk-averse
driver, who makes decisions based solely on $\overline{u}^m_t$.
Similarly, $\pi^1$ and $\pi^{1/2}$ model risk-seeking and risk-neutral drivers, \neww{respectively}.
Under certain assumptions \neww{ bounding the modulus of continuity of functions $c_A, c_B, \ldots$, \emph{cf.} \cite{marecek2015signaling},} one can show that this results in a stable behaviour of the system.

Considering that \neww{\emph{any}} randomisation may \neww{be} undesirable, \cite{marecek2016signaling} suggested broadcasting a deterministically chosen interval for each route.
In one such approach, called $r$-extreme \cite{marecek2016signaling}, one simply broadcasts the maximum and minimum travel time \neww{within} a time window of  $r$ most recently observed travel times.
In another variant, called exponential smoothing \cite{Epperlein2016}, one broadcasts a weighted combination of the current travel-time \neww{and past travel times, alongside a weighted combination of the} current variance \neww{of the travel times} and the previously sent information \neww{about the variance}.
Under \neww{some additional} assumptions, one can analyse the resulting stochastic \neww{(delay)} difference equations\neww{:}
Using \neww{results} developed \neww{in the} theory of iterated random functions \cite{IFS}, \cite{marecek2016signaling} show that the $r$-extreme schema \neww{yields ergodic behaviour when the distribution 
of types of drivers changes over time in a memory-less fashion}.
\cite{Epperlein2016} extended the result to populations, whose evolution is governed by a Markov chain, \neww{which allows, \emph{e.g.}, for different distributions
at different times of the day, such as at night, during the morning and afternoon peaks, and all other times.}
In Table~\ref{tab:int-signals}, we present an overview of these schemata.

We should like to stress that the above is not a comprehensive overview of related work.
We refer to \cite{arnott1991does,BENAKIVA1991251,bonsall1992influence,arnott1993structural,emmerink1996information}
for pioneering studies in the field as well as to \cite{bottom2000consistent,papageorgiou2007its} for extensive, book-length 
overviews of further related work.

\section{Distributionally robust optimisation}

In this paper, we suggested broadcasting a deterministically chosen interval for each route, 
where the deterministic choice is based on optimisation over subintervals of the 
\neww{interval given by the minimum and maximum over a time window of a finite, fixed length $r$.} 
For $1 < r < t$, we define $s_{t} =
(\underline u_{t}^1, \overline u_{t}^1, \underline u_{t}^2, \overline u_{t}^2, \ldots, \underline u_{t}^M, \overline u_{t}^M, )$ to be \emph{$r$-supported}, whenever 
\begin{align}
\label{eq:supported1}
  \min_{j=t-r,\ldots,t} \{ c_m(n^m_j) \} \le \underline u_{t}^m < \overline u_{t}^m \leq \max_{j=t-r,\ldots,t} \{ c_m(n^m_j) \}.
\end{align}
Notice that $r$-extreme $s_t$ is a special case of $r$-supported $s_t$.
To study the effects \neww{of} broadcasting $r$-supported $s_t$, we need to formalise the 
\neww{model} of the population. Clearly, one can start with:

\begin{assumption}[Full Information]\label{as:fixed}
Let us assume that $\Omega$ is a finite set. Further, let us assume the number of drivers of type $\omega$ at time $t+1$ is $N \mu_{t+1}(\omega)$ and that $N \mu_{t+1}(\omega)$ 
is known to the network operator at time $t$.
\end{assumption}

Assumption~\ref{as:fixed} is very restrictive.  Instead, we may want to assume \neww{that} $\mu_t$ \neww{are} independently identically distributed (i.i.d.) samples of \neww{a} random variable. 
\footnote{\neww{One could go further still and assume time-varying distributions of $\mu_t$, or more general structures, yet. We refer to \cite{Epperlein2016} for an example, but note that such assumptions do not allow for the efficient application of methods of computational optimisation, in general. In this paper, we hence consider the i.i.d. assumption.}} 
In the tradition of robust optimisation \cite{Soyster1974}, one could \neww{assume that a support of the random variable is known and} optimise \neww{social cost} over \neww{all possible distributions of the random variable
with the given support}. That approach, however, tends to produce
overly conservative solutions, when it produces any feasible solutions at all.  
In the tradition of distributionally robust optimisation \neww{(DRO)} \cite{bertsimas2010models,delage2010distributionally},
one could \neww{assume that a certain number of moments of the random variable are known and} optimise \neww{social cost} over \neww{all possible distributions of the random variable with the given moments}. \neww{We suggest to use DRO with the first} two moments:\\[1mm]

\begin{assumption}[Partial Information]\label{as:pop}
Let us assume that $\Omega$ is a finite set. Let us assume the number of drivers of type $\omega$ at time $t+1$ is $N \mu_{t+1}(\omega)$, but that the distribution of $\mu_{t+1}$ is unknown
at time $t$, except for the first two moments of the distribution of $\mu_{t+1}$, denoted $E, Q$:
\begin{align}
E &=
\begin{pmatrix}
E_1  \\
E_2  \\
\vdots \\
E_{|\Omega|} \\
\end{pmatrix}
= \E \begin{pmatrix}
\mu_{t+1}(1)  \\
\vdots \\
\mu_{t+1}(\abs{\Omega}) \\
\end{pmatrix} \\
Q &= 
\begin{pmatrix}
Q_{11} & \ldots & Q_{1|\Omega|} \\
\vdots & & \vdots \\
Q_{|\Omega|1} & \ldots & Q_{|\Omega||\Omega|}\\
\end{pmatrix}
= 
\E \left[ \begin{pmatrix}
\mu_{t+1}(1)  \\
\mu_{t+1}(2)  \\
\vdots \\
\mu_{t+1}(|\Omega|) \\
\end{pmatrix}
\begin{pmatrix}
\mu_{t+1}(1)  \\
\mu_{t+1}(2)  \\
\vdots \\
\mu_{t+1}(|\Omega|) \\
\end{pmatrix}^T
\right] 
\end{align}
 and let us assume $E, Q$ are known to the network operator at time $t$.
\end{assumption}

Notice that Assumption~\ref{as:pop} is much more reasonable \neww{than Assumption~\ref{as:fixed}. 
The authorities can compute an unbiased estimate of} the first two moments 
using \neww{readily-available} statistical 
estimation techniques \cite{toledo2004calibration,vaze2009calibration}. 
\neww{In contrast, ascertaining the actual realisation of the random variable in 
real time seems impossible, and 
estimating more than two moments of a multi-variate random variable remains a challenge,
as the requisite number of samples grows exponentially with the order of the moment,
which in turn makes the computations prohibitively time consuming.
In short, we believe that Assumption~\ref{as:pop} presents a suitable trade-off between realism
and practicality.  
}

Next, one needs to decide on the objective, which should be optimised. 
Clearly, even a finite-horizon approximation of the accumulated social cost is a challenge.
Beyond that, we can show a yet stronger negative result:\\[2mm]

\begin{proposition}[Undecidability]
\label{undecidable}
   Under Assumption~\ref{as:fixed},
   there exist $c_A, c_B$, and an initial $s_1$ broadcast, such that it is undecidable whether iterates $s_t \in \R^n$, $n \ge 2$ induced by 
   policies $\pi^\omega$ responding to intervals broadcast converge to a point  
   $s_T \in \R^n$ from $s_1$, such that $s_t$ for all $t > T$ is equal to $s_T$.
\end{proposition}

The proof is based on the results of \cite{blondel2001deciding,koiran1994computability} 
 that given piecewise affine function $g : R^2 \to R^2$ and an initial point $x_0 \in R^2$, 
it is undecidable whether iterated application $g \ldots g(x_0)$ reaches a fixed point, eventually, and the fact we make no assumptions about the functions $c_m$.  
\neww{Although Proposition~\ref{undecidable} does not rule out weak convergence guarantees in the measure-theoretic sense under Assumption~\ref{as:fixed}, for instance, some assumptions concerning the functions $c_m$ do simplify the matters considerably.}

\neww{To formulate such an assumption,} observe that the function $g$ corresponds to a composition of the social cost \eqref{eqn:socialtravel time} and the policy \eqref{policylambda}. In particular:
  \begin{align}
    \underline u^m_{t+1} & = \min \{\underline u^m_t, c_m(n^m_t)\},\notag \\
    \overline u^m_{t+1} &= \max \{\overline u^m_t, c_m(n^m_t)\}, \notag 
  \end{align}
wherein one applies $c_m$ to values of $n^m_t$:
\neww{
  \begin{align}
    n^m_t &= \sum_i 1_{(a^i_t = A_m)}  \\
    &= \sum_i \sum_{\omega \in \Omega} 1_{(a^i_t = A_m \mid \textrm{ driver } i \textrm{ is of type }
      \omega
      )} \mu_t(\omega) \notag \\
    &= N \sum_{\omega \in \Omega} 1_{ \bigwedge_{s \neq r} ( \omega
      \underline u^m_t + (1-\omega) \overline u^m_t < \omega
      \underline u^s_t + (1-\omega) \overline u^s_t )}
    \mu_t(\omega)\label{eq:29},
  \end{align}
whereby one obtains $\underline u^m_{t+1}, \overline u^m_{t+1}$ as a function of $\underline u^m_{t}, \overline u^m_{t}$. 
We refer to the proof of Theorem 1 in \cite{marecek2016signaling} for a detailed discussion of this signal-to-signal mapping
and properties of $\mu_t$.
}

\neww{One may hence obtain a signal-to-signal mapping $g$ of more desirable properties} by restricting oneself to a particular class of $c_m$, and hence to a particular class of social costs \eqref{eqn:socialtravel time}.
\neww{In particular, we restrict ourselves to}:\\[2mm]

\begin{proposition}[$C$ is Difference of Convex]
\label{prop:DCA}
   For any functions $c_m$ convex on $[0, 1]$,
   there exist solvers for the minimisation of the unconstrained social cost $C$ (cf. Eq.~\ref{eqn:socialtravel time}), 
   with guaranteed convergence to a stationary point.
\end{proposition}

Using a wealth of results \cite{DC} on the optimisation of DC (``difference of convex'') functions,
we can show:

\begin{proposition}[The Full Information Optimum]
\label{thm:Computable}
   Under Assumption~\ref{as:fixed},
   a stationary point of: 
   \neww{ 
\begin{align}
\label{eq:Computable}
  \min_{ s_t = (\underline{u}^m_t, \overline{u}^m_t, 1 \le m \le M) \in P(S_t, \Omega)} C(n_t)
\end{align}
   \neww{can be computed up to any fixed precision} in finite time, 
   where $P \subseteq \R^{2M}$ 
   is the set of $r$-supported signals \eqref{eq:supported1}
   and functions $c_m, 1 \le m \le M$ are convex on $[0, 1]$.\\[2mm]
}
\end{proposition}

\begin{proof}
  \label{proof3}
  Let us introduce an auxiliary indicator variable and a non-negative
  continuous variable:
  \begin{align*}
    x_{t,m}^\omega = & 1_{[\pi^{\omega}(s_t) = m]} = \begin{cases}
      \; 1 & \; \text{ if } \omega \text{ selects action } m \text{ at time } t \\
      \; 0 & \; \text{ otherwise }
    \end{cases} \\
    \underline y_{t,i,j}^\omega = &
    \begin{cases}
      \; - g^\omega_{i,j} & \text{ if } g^\omega_{i,j} < 0 \\
      \; 0 & \text{otherwise}
    \end{cases} \\ 
    \overline y_{t,i,j}^\omega = &
    \begin{cases}
      \; g^\omega_{i,j} & \text{ if } g^\omega_{i,j} \ge 0 \\
      \; 0 & \text{otherwise}
    \end{cases}
  \end{align*}
  where \neww{$g^\omega_{i,j} \triangleq \omega \underline{u}_t^i + (1-\omega)
  \overline{u}_t^i - \omega \underline{u}_t^j - (1-\omega)
  \overline{u}_t^j$}.  See that $n_t^m = \sum_{\omega \in \Omega}
  x_{t,m}^\omega n_{t}(\omega)$. 
  Sometimes, we use $x_{t}$ to denote a matrix of $x_{t,m}^\omega$ for all $\omega \in \Omega$,
  $1 \le m \le M$.

It is easy to show there exist a lifted
  polytope $P'$ such that:
  \begin{align}
    \label{eq:Obj}
    \min_{
      s_t 
      \in \R^{2M}} 
    \sum_{m=1}^M \frac{n^m_t}{N} \cdot c_m(n^m_t), \notag \\  s \in P(S_t,\Omega) \\
    = \min_{
      s_t 
      \in \R^{2M},
      \underline y_t, \overline y_t \in \R^{M(M-1)|\Omega|} x_t \in \{0, 1
      \}^{M|\Omega|} } 
   \sum_{m=1}^M \frac{n^m_t}{N} \cdot c_m(n^m_t) \notag \\ (s_t, x_t, \underline y_t, \overline
    y_t ) \in P'(S_t,\Omega),
  \end{align}
  The definition of the polytope $P'$ depends on the policies defined
  by $\Omega$ and the history of signals $S_t$. Specifically:
  \begin{align}
    \omega \underline{u}^i_{t} + (1-\omega) \overline{u}^i_{t} -
    \omega \underline{u}^j_{t} - (1-\omega) \overline{u}^j_{t}
    \leq & \overline y_{t,i,j}^\omega  & \neww{\forall i, j, t, \omega} \\
    \omega \underline{u}^j_{t} + (1-\omega) \overline{u}^j_{t} -
    \omega \underline{u}^i_{t} - (1-\omega) \overline{u}^i_{t}
    \leq & \underline y_{t,i,j}^\omega  & \neww{\forall i, j, t, \omega}  \\
    - Z (1 - x_{t,m}^\omega) \leq &  \underline y_{t,m,i}^\omega \leq Z (1 - x_{t,m}^\omega)  & \neww{\forall t, m, \omega}  \\
    - Z x_{t,m}^\omega \leq &  \overline y_{t,m,i}^\omega \leq Z x_{t,m}^\omega  & \neww{\forall t, m, \omega}  \\
    \overline u_{t}^m \geq & \underline u_{t}^m \geq \min_{j=t-r,\ldots,t} \{ c_m(n^m_j) \}  & \neww{\forall m, r, t}  \\
    & \overline u_{t}^m \leq \max_{j=t-r,\ldots,t} \{ c_m(n^m_j) \} & \neww{\forall m, r, t} \\
    \sum_{m = 1}^{M} x_{t,m}^{\omega} =& 1 & \neww{\forall t, \omega} \\
    \underline y_\omega, \overline y_\omega \geq & 0 & \neww{\forall \omega}
  \end{align}
  where the $\max, \min$ operators are applied to 
  \neww{the revealed realisations of the random variable $n^m_j$, and hence  
  yield constants, rather than bi-level structures.} Further, $Z$ is a sufficiently large constant, \eg 
  $$
  \max_{\substack{m = 1,2,\ldots,M, m'\in\{1,2, \ldots, M \}\setminus
    \{m\}\\ \underline{u}^m_{t},\overline{u}^m_{t},\underline{u}^{m'}_{t},\overline{u}^{m'}_{t}}}
  \{ |\omega \underline{u}^m_{t} + (1-\omega) \overline{u}^m_{t} -
  \omega \underline{u}^{m'}_{t} - (1-\omega) \overline{u}^{m'}_{t}| \}
  \notag \leq \max_x \{ C(x) \}.  $$

  The integer component can be solved by branching, whereby the
  Lagrangian gives us an unconstrained relaxation of the original
  problem.  Hence, by Proposition~\ref{prop:DCA}, the \neww{stationary point can be computed} 
   up to any precision in finite time.
\end{proof}

\begin{proposition}[The Distributionally Robust Optimum]
   Under Assumption~\ref{as:pop},
   let us consider functions $c_m, 1 \le m \le M$ convex on $[0, 1]$ and    
\neww{
\begin{align}
\label{eq:robust}
  \min_{ (\underline{u}^m_t, \overline{u}^m_t, 1 \le m \le M) \in P(S_t, \Omega)} \sup_{D\sim (E, Q)} 
  C\left( \E_D n_t \right)
\end{align}
}
   where $D\sim (E, Q)$ in the inner optimisation problem suggests optimisation over the infinitely many distribution functions 
   of $\Omega$ with the first two moments of Assumption~\ref{as:pop},
  and $P \in \R^{2M}$ 
  is the set of $r$-supported signals \eqref{eq:supported1}.
  A stationary point of the distributionally robust optimisation problem \eqref{eq:robust} 
  \neww{can be computed up} to any \neww{fixed} precision in finite time.\\[2mm]
\end{proposition}

\begin{proof}
  \label{proof4}


  Notice that we can reformulate the problem \eqref{eq:robust} as an
  integer semidefinite program by the introduction of a new decision
  variable $W$ in dimension $|\Omega| \times |\Omega|$, vector $w_{\omega}
  \in \R^{|\Omega|}$, and scalar $q_{\omega}$, in addition to the variables
  introduced in the proof of Proposition \ref{eq:Computable}: 
  \begin{align}
  \min_{ (\underline{u}^m_t, \overline{u}^m_t, 1 \le m \le M) \in P(S_t, \Omega)} \sup_{D \sim (E, Q)} & 
 C\left( \E_D n_t \right) \\
=  \min_{ \substack{
        s_t 
        \in \R^{2M} \\
      x_t \in \{0, 1
      \}^{M|\Omega|} 
 } }  
\sup_{ D \sim (E, Q)} &
    \sum_{m=1}^M \left( \sum_{\omega \in \Omega} \left( x_{t,m}^\omega \E_D \mu_t(\omega)  \right) \cdot c_m \left( \sum_{\omega \in \Omega} x_{t,m}^\omega N \E_D \mu_t(\omega)  \right) \right) \\
\textrm{s.t. }
     & \E_D \mu_t = E \notag \\
     & \E_D \left[ \mu_t \mu_t^T \right] = Q \notag \\
     & s \in P'(S_t, \Omega) \notag \\
    = \min_{ \substack{
        s_t 
        \in \R^{2M}, \\
      \underline y_t, \overline y_t \in \R^{M(M-1)|\Omega|} \\
  x_t \in \{0, 1 \}^{M|\Omega|} \\ 
        W_{\omega} \in \R^{|\Omega| \times |\Omega|}\\ w_{\omega} \in \R^{|\Omega|},
        q_{\omega} \in \R } } 
 &
    \sum_{m=1}^M \frac{\sum_{\omega \in \Omega} \left( x_{t,m}^\omega e_\omega w_\omega \right)}{N} \cdot c_m \left( \sum_{\omega \in \Omega} \left( x_{t,m}^\omega e_\omega w_\omega \right) \right) \\
\textrm{s.t. }
 &     \sum_{\omega \in \Omega} \begin{pmatrix}
      W_\omega & w_\omega \\
      w_\omega^T & q_\omega
    \end{pmatrix} = \begin{pmatrix}
      Q & E \\
      E^T & 1
    \end{pmatrix} \notag \\
 &     \begin{pmatrix}
      W_\omega & w_\omega \\
      w_\omega^T & q_\omega
    \end{pmatrix} \succeq 0 \quad \forall \omega \in \Omega \notag \\
 &     (s_t, x_t, \underline y_t, \overline y_t ) \in
    P'(S_t,\Omega),\notag
  \end{align}
  where $e_{\omega}$ are vectors with only the $\omega^\text{th}$
  entry of $1$ and others $0$.  The first equality
  follows from the definition of $C$ \eqref{eqn:socialtravel time}.
  The second equality follows from the work of Bertsimas et al. \cite{bertsimas2010models}
  on minimax problems, and specifically from Theorem 2.1 therein.
  Although Theorem 2.1 does not consider integer variables explicitly,
  it is easy to see that for each of the $2^{M|\Omega|}$ possible integer values of $x_t$, the equality holds, and hence it holds generically.
  See also the lucid treatment of Mishra et al. \cite{Mishra2012}.

  Computationally, one can apply branching to the integer variables $x_t$, as in the proof of Proposition \ref{thm:Computable}, which leaves one with a semidefinite
  program with a non-convex objective.  
  There, one can formulate the
  augmented Lagrangian, which is non-convex, but well-studied \cite{Stingl2009,kovcvara2003pennon,lanckriet2009convergence,yen2012convergence}.
  For instance, it can be reformulated to a ``difference of convex'' form and Proposition~\ref{prop:DCA} can be applied.
  Let us multiply $C(\cdot)$ by $N$ to study the $2 M$ terms one by one. 
  We want to show that the rest is a sum of a convex and concave terms.
  Let us see that for $i = 1, 2, \ldots, M - 1$,
  we have the term $n_t^i c_i(n_t^i)$,
  which is convex in $n_t^i$, considering that for convex and non-decreasing $g$ and convex $f$, we know $g(f(x))$ is convex.
  For $i = M$, we have the terms
  $c_i( 1-\sum_{i = 1}^{M-1} n_t^i)$
  and a $M - 1$ terms from 
  $- (\sum_{i = 1}^{M-1} n_t^i) c_i( 1-\sum_{i = 1}^{M-1} n_t^i)$.
  Considering that convexity is preserved by affine substitutions of
  the argument, the former term is convex for the affine subtraction and convex
  $c_i$.
  Considering the additive inverse of a 
  convex function is a concave function, we see $ - n_t^i n_t^M(\cdot)$ is
  concave.  
  The proposition
  follows from the following Proposition~\ref{prop:DCA}.
\end{proof}

Alternatively, one may consider
polynomial functions $c_m$, where the minimum of the social cost $C$ can be \neww{computed} up
  to any \neww{fixed} precision in finite time by solving a number of instances of semidefinite 
  programming (SDP).

\section{A computational illustration}

For optimisation problems such as \eqref{eq:Computable} and \eqref{eq:robust}, there are solvers based on sequential convex
programming with known rates of convergence
\cite{lanckriet2009convergence,yen2012convergence}.  
In our computational experiements, we have extended a
sequential convex programming solver of Stingl et
al. \cite{Stingl2009}, which handles polynomial semidefinite programming of \eqref{eq:robust},
to handle mixed-integer polynomial semidefinite programming.
Specifically, Stingl et al. \neww{replace} nonlinear objective functions
by block-separable convex models, 
following the approach of Ben-Tal and Zhibulevsky \cite{ben1997penalty} and Ko{\v
  c}vara and Stingl \cite{kovcvara2003pennon}.

In our experiments, we have considered the same set-up as in \cite{marecek2016signaling},
where $M=2$ and 
two Bureau of Public Roads (BPR) functions are used for the 
costs, as presented in Figure~\ref{fig:ushaped5settings}.
The population is given by 
$\Omega = \{ 0, 0.5, 1, \textrm{Uniform}(0, 1), \textrm{Uniform}(0, 1) \}$, 
the initial signal is $s_1 = (0.5, 1, 0.6, 0.9)$, and 
$\kappa = 0.15$,
$\mu_t(\omega) \sim \textrm{Uniform}(1/5 - \kappa, 1/5 + \kappa) \; \forall \omega \in \Omega, t > 1$, 
and $N=30$. 
\neww{
These settings have been chosen both for the simplicity of reproduction as well as to allow for comparison 
with plots presented in \cite{marecek2015signaling,marecek2016signaling}.
}

\neww{
Figure~\ref{fig:ushaped5simulations2} illustrates the cost $\{C(n_t)\}$ over time $1 \le t \le 20$,
 for three lengths $r$ of the look-back, $r = 1$ (top), $r = 3$ (middle), $r = 5$ (bottom), 
 with error bars at one standard deviation capturing the variability over the sample paths.
It seems clear that the $r$-supported scheme (in dark blue, Eq. \ref{eq:robust}) is only marginally worse than  
the full-information optimum (in red, Eq. \ref{eq:Computable}), which is ``pre-scient'' and hence impossible to operate in the real-world.
Also, it seems clear that for low values of $r$, there is not enough data to estimate the second moments,
and hence the use of the first moment (in green) behaves similarly to the use of the first two moments (in dark blue).
Both compared to the use of the first moment and to the previously proposed $r$-extreme scheme (in light blue),
the $r$-supported scheme yields costs with less prominent extremes,
even after averaging over the sample paths. 
}

\neww{
Further, Figure~\ref{fig:ushaped5simulations1} illustrates the the process $\{C(n_t)\}$ averaged over $1 \le t \le 20$ for varying $r$, again with error bars at one standard deviation.
It shows that employing 
 $r$-supported scheme (in dark blue, Eq. \ref{eq:robust})
 allows for a reduction of the social cost,
 when compared to $r$-extreme singalling (in light blue), 
 across a range
of the length $r$ of the look-back interval.
Again, it seems clear that the $r$-supported scheme (in dark blue) is only marginally worse than  
the full-information optimum (in red, Eq. \ref{eq:Computable}).
}

\begin{figure}[t!]
  \includegraphics[clip=true,width=0.49
  \textwidth]{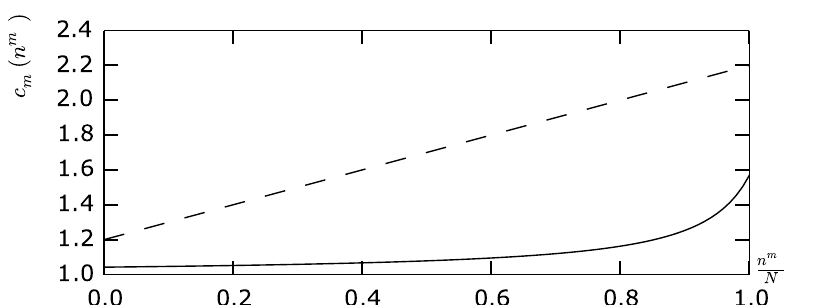}
  \includegraphics[clip=true,width=0.49
  \textwidth]{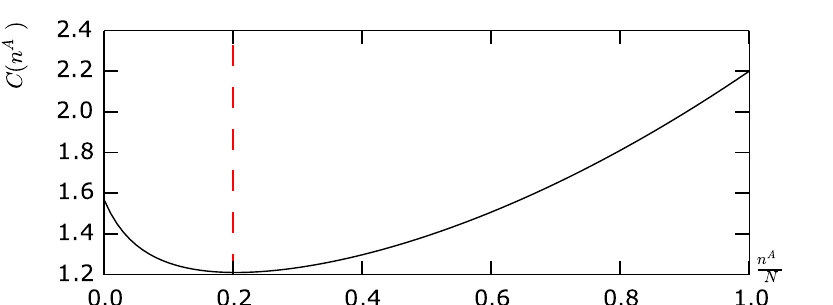}\\[1mm]
                
  \caption{ A trivial example with $M=2$. Left: Cost functions $c_1(x)
    \triangleq 2 (1 + 3.6 x^4)$ and $c_2(y) \triangleq 5 (1 + 0.8 y^2)$.
    Right: The corresponding social cost.
}
  \label{fig:ushaped5settings}
\end{figure}

\begin{figure}[t!]
\begin{centering}
\centering
  \includegraphics[clip=true,width=0.8\textwidth]{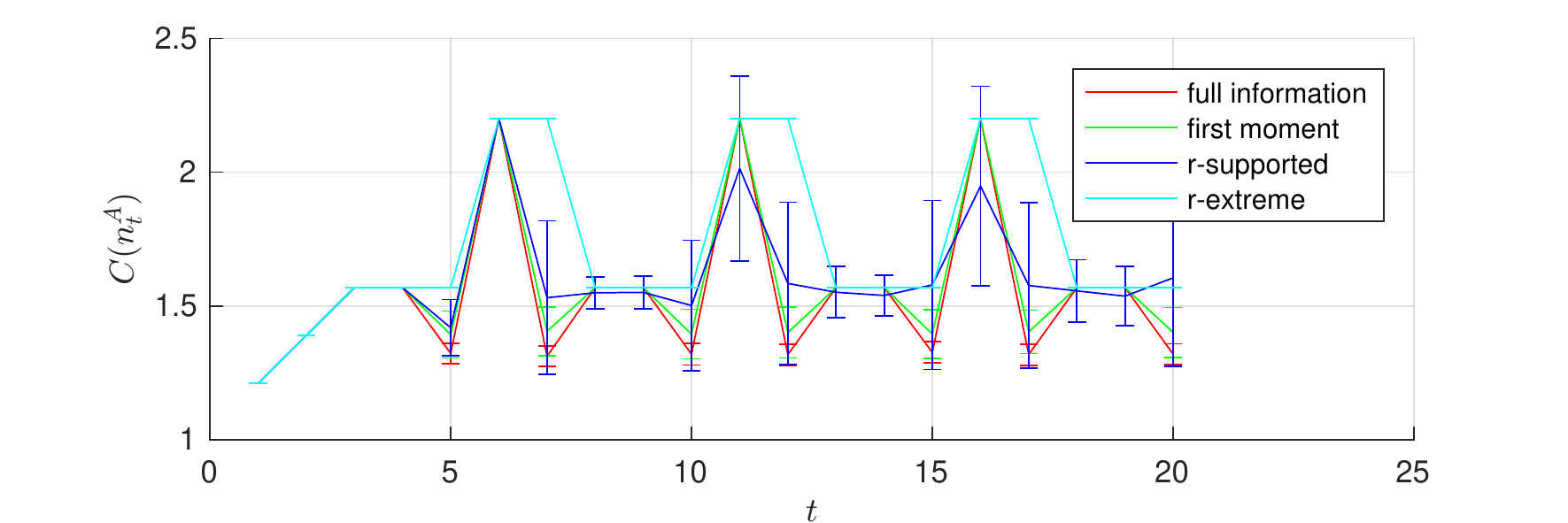}\\
  \includegraphics[clip=true,width=0.8\textwidth]{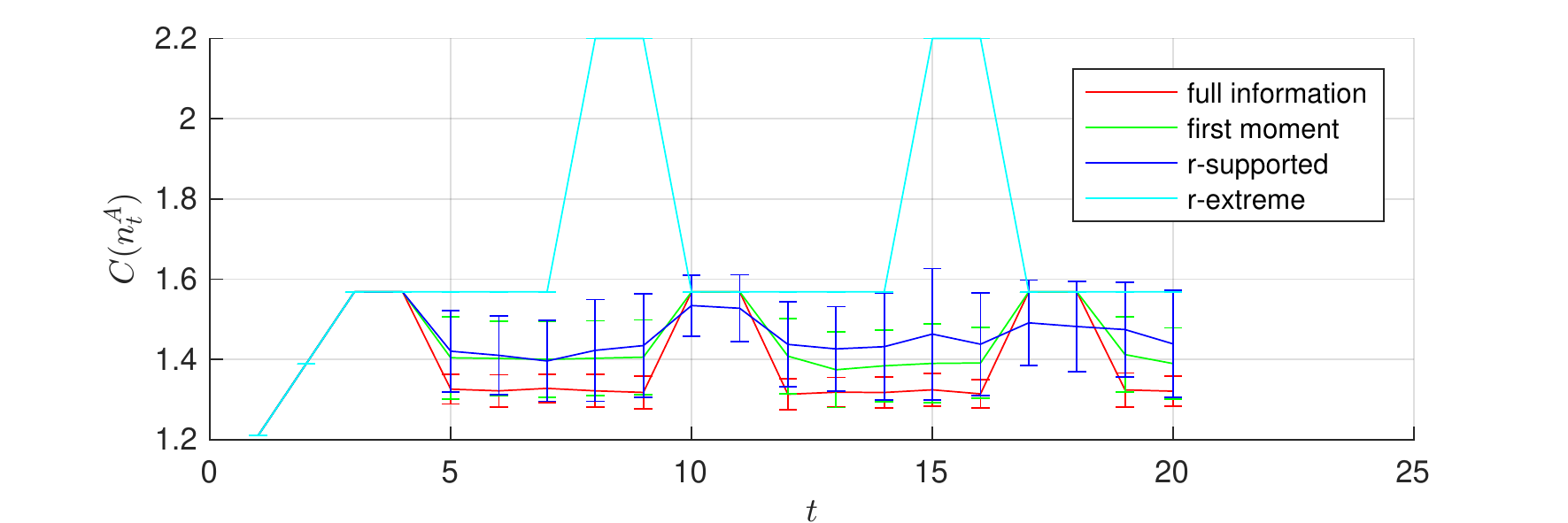}\\
  \includegraphics[clip=true,width=0.8\textwidth]{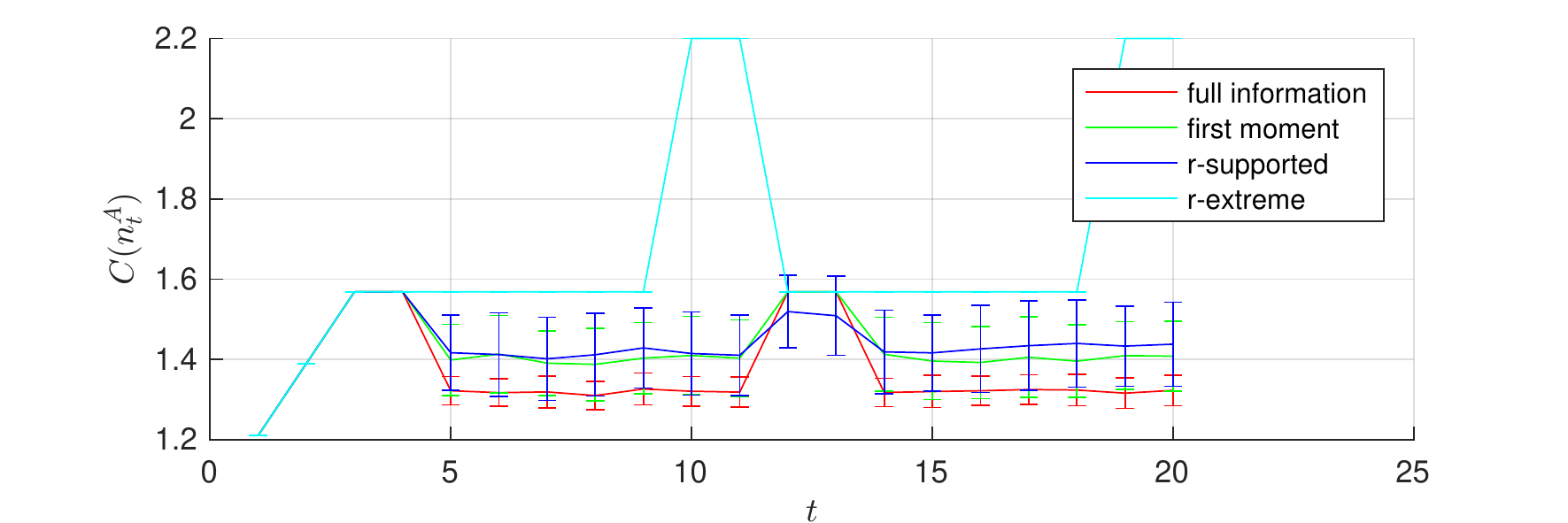}\\
\end{centering}
                \caption{ $r$-supported scheme (in dark blue) compared to the 
``pre-scient'' full information optimum (in red), the use of the first moment (in green), and the previously proposed $r$-extreme scheme (in light blue): 
The process $\{C(n_t)\}$ over time for $r = 1$ (top), $r = 3$ (middle), $r = 5$ (bottom) with error bars at one standard deviation across the sample paths.
}
  \label{fig:ushaped5simulations2}
\end{figure}

\begin{figure}[t!]
\begin{centering}
\centering
  \includegraphics[clip=true,width=0.8\textwidth]{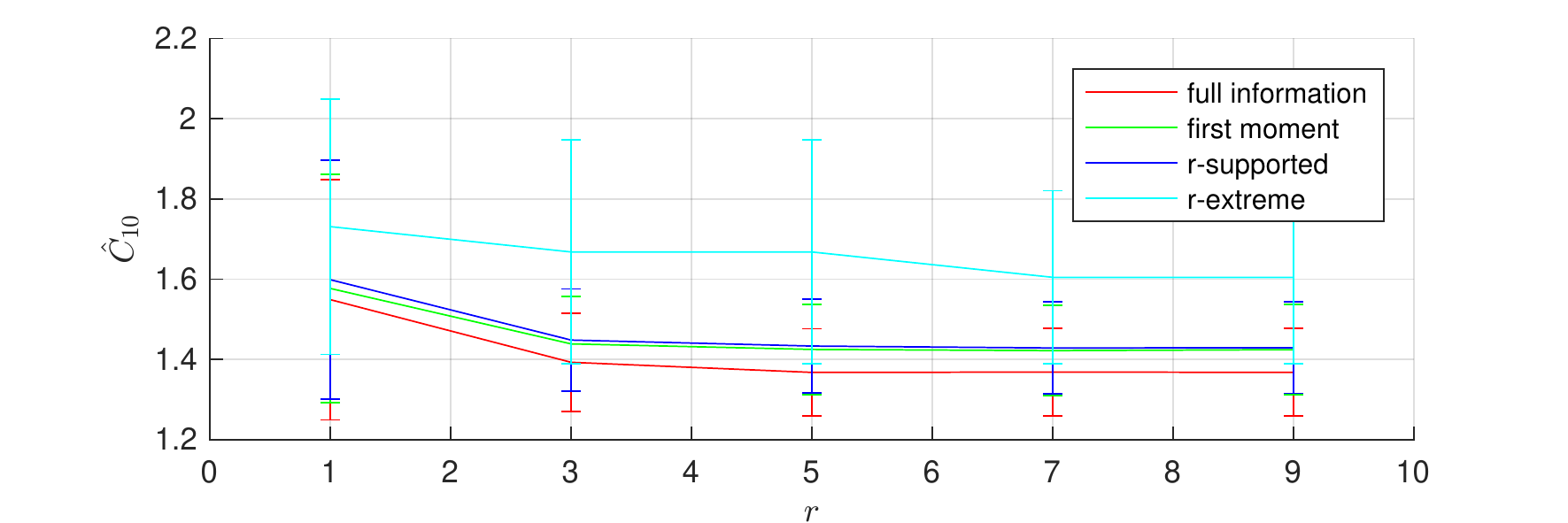}\\
\end{centering}
                \caption{ $r$-supported scheme (in dark blue) compared to the 
``pre-scient'' full information optimum (in red), the use of the first moment (in green), and the previously proposed $r$-extreme scheme (in light blue): 
The process $\{C(n_t)\}$ averaged over $1 \le t \le 10$ for varying $r$ with error bars at one standard deviation across the sample paths.
}
  \label{fig:ushaped5simulations1}
\end{figure}

Finally, we note that the for 
The stationary point \eqref{eq:Computable} can be \neww{computed up} to \neww{precision} $10^{-6}$ in about 15 seconds on a basic laptop with Intel i5-2520M, 
  although the run-time does increase with the number of routes. 
\neww{This is much more than the run-time of the previously proposed $r$-extreme scheme. An efficient implementation of the $r$-supported scheme remains a major challenge for future work.}

\section{Conclusions}

In conclusion, \neww{there are multiple ways of introducing} ``uncertainty'' into \neww{the behaviour of the road user in terms of the route choice.} 
Previously, the addition of zero-mean noise with a positive variance $\sigma$ \cite{marecek2015signaling}, broadcasting intervals such as $(\delta,\gamma)$ intervals \cite{marecek2015signaling} and 
 $r$-extreme intervals (minima and maxima over a time window of size $r$) \cite{marecek2016signaling}, and intervals based on exponential smoothing \cite{Epperlein2016},
 have been shown result in the distribution of drivers over the road network converging over time, 
 under a variety of assumptions about the evolution of the population over time.
This paper studied the optimisation of the social cost over sub-intervals within the minima
  and maxima over a time window of size $r$, under a variety of assumptions.

This paper is among the first applications of distributionally robust optimisation (DRO) 
in transportation research.
while other recent work considered its use in stochastic traffic assignment \cite{Ahipasaoglu2015},
where it presents a tractable alternative to multinomial probit \cite{Mishra2012},
and in traffic-light setting \cite{Liu2015536}.
We envision there will be a wide variety of further studies,
 once the power of DRO is fully appreciated in the community.

This work opens a number of questions in cognitive science, multi-agent systems, 
artificial intelligence, and urban economics.  
How do humans react to intervals, actually?  
How to invest in transportation infrastructure, knowing that information provision can be co-designed to suit the infrastructure?
Future technical work may include the study of variants of the proposed scheme, such as 
broadcasting $s_t$ such that 
\begin{align}
  \avg_{j=t-r,\ldots,t} \{ c_m(n^m_j) \} \geq \underline u_{t}^m &\geq \min_{j=t-r,\ldots,t} \{ c_m(n^m_j) \},\\
  \avg_{j=t-r,\ldots,t} \{ c_m(n^m_j) \} \leq \overline u_{t}^m &\leq \max_{j=t-r,\ldots,t} \{ c_m(n^m_j) \},
\end{align}
where $\avg$ denotes the average.
One could also employ risk measures such as value at risk (VaR) and conditional value at risk (CVaR) 
 for a given coefficient $\alpha$ and distribution function $L$ with support $\{ c_m(n_t^m) \}_t$,
 as suggested in  Table~\ref{tab:int-signals}.
Further studies of (weak) convergence properties \cite{marecek2016signaling,Epperlein2016}, including the rates of convergence,
 and further developments of the population dynamics \cite{Epperlein2016} would also be most interesting.
Beyond transportation, one could plausibly employ similar techniques in related resource-sharing problems (e.g., ad keyword auctions,
 dynamic pricing in power systems, announcements of emergency evacuation routes)
in order to improve the variants of social costs therein.

\section{Acknowledgement} This research received funding from the European Union Horizon 2020 Programme (Horizon2020/2014-2020) under grant agreement number 688380.
Robert Shorten has been funded by Science Foundation Ireland under grant number 11/PI/1177.

\bibliographystyle{is-plain}
\bibliography{traffic}

\begin{thebibliography}{10}
\ifx \showCODEN  \undefined \def \showCODEN #1{CODEN #1}  \fi
\ifx \showISBN   \undefined \def \showISBN  #1{ISBN #1}   \fi
\ifx \showISSN   \undefined \def \showISSN  #1{ISSN #1}   \fi
\ifx \showLCCN   \undefined \def \showLCCN  #1{LCCN #1}   \fi
\ifx \showPRICE  \undefined \def \showPRICE #1{#1}        \fi
\ifx \showURL    \undefined \def \showURL {URL }          \fi
\ifx \path       \undefined \input path.sty               \fi
\ifx \ifshowURL \undefined
     \newif \ifshowURL
     \showURLtrue
\fi

\bibitem{Ahipasaoglu2015}
Selin~Damla Ahipasaoglu, Rudabeh Meskarian, Thomas~L. Magnanti, and Karthik
  Natarajan.
\newblock Beyond normality: A cross moment-stochastic user equilibrium model.
\newblock {\em Transport. Res. B - Meth.}, 81, Part 2:\penalty0 333 -- 354,
  2015.
\newblock \showISSN{0191-2615}.

\bibitem{DC}
Le~Thi~Hoai An and Pham~Dinh Tao.
\newblock The dc (difference of convex functions) programming and dca revisited
  with dc models of real world nonconvex optimization problems.
\newblock {\em Ann. Oper. Res.}, 133\penalty0 (1):\penalty0 23--46.
\newblock \showISSN{1572-9338}.

\bibitem{arnott1991does}
Richard Arnott, Andre De~Palma, and Robin Lindsey.
\newblock Does providing information to drivers reduce traffic congestion?
\newblock {\em Transport. Res. A: Gen.}, 25\penalty0 (5):\penalty0 309--318,
  1991.

\bibitem{arnott1993structural}
Richard Arnott, Andre De~Palma, and Robin Lindsey.
\newblock A structural model of peak-period congestion: A traffic bottleneck
  with elastic demand.
\newblock {\em Amer. Econ. Rev.}, pages 161--179, 1993.

\bibitem{BENAKIVA1991251}
Moshe Ben-Akiva, Andre~De Palma, and Kaysi Isam.
\newblock Dynamic network models and driver information systems.
\newblock {\em Transport. Res. A: Gen.}, 25\penalty0 (5):\penalty0 251 -- 266,
  1991.
\newblock \showISSN{0191-2607}.

\bibitem{ben1997penalty}
Aharon Ben-Tal and Michael Zibulevsky.
\newblock Penalty/barrier multiplier methods for convex programming problems.
\newblock {\em SIAM J. Optimiz.}, 7\penalty0 (2):\penalty0 347--366, 1997.

\bibitem{bertsimas2010models}
Dimitris Bertsimas, Xuan~Vinh Doan, Karthik Natarajan, and Chung-Piaw Teo.
\newblock Models for minimax stochastic linear optimization problems with risk
  aversion.
\newblock {\em Math. Oper. Res.}, 35\penalty0 (3):\penalty0 580--602, 2010.

\bibitem{blondel2001deciding}
Vincent~D Blondel, Olivier Bournez, Pascal Koiran, Christos~H Papadimitriou,
  and John~N Tsitsiklis.
\newblock Deciding stability and mortality of piecewise affine dynamical
  systems.
\newblock {\em Theor. Comput. Sci.}, 255\penalty0 (1):\penalty0 687--696, 2001.

\bibitem{bonsall1992influence}
Peter Bonsall.
\newblock The influence of route guidance advice on route choice in urban
  networks.
\newblock {\em Transportation}, 19\penalty0 (1):\penalty0 1--23, 1992.

\bibitem{bottom2000consistent}
Jon~Alan Bottom.
\newblock {\em Consistent anticipatory route guidance}.
\newblock PhD thesis, Massachusetts Institute of Technology, 2000.

\bibitem{delage2010distributionally}
Erick Delage and Yinyu Ye.
\newblock Distributionally robust optimization under moment uncertainty with
  application to data-driven problems.
\newblock {\em Oper. Res.}, 58\penalty0 (3):\penalty0 595--612, 2010.

\bibitem{IFS}
Persi Diaconis and David Freedman.
\newblock Iterated random functions.
\newblock {\em SIAM Rev.}, 41\penalty0 (1):\penalty0 45--76, 1999.

\bibitem{emmerink1996information}
Richard~HM Emmerink, Erik~T Verhoef, Peter Nijkamp, and Piet Rietveld.
\newblock Information provision in road transport with elastic demand: A
  welfare economic approach.
\newblock {\em J. Transp. Econ. Pol.}, pages 117--136, 1996.

\bibitem{Epperlein2016}
Jonathan Epperlein and Jakub Mare{\v{c}}ek.
\newblock Resource allocation with population dynamics.
\newblock {\em arXiv pre-print, arXiv:1604.03458}, 2016.

\bibitem{kovcvara2003pennon}
Michal Ko{\v{c}}vara and Michael Stingl.
\newblock Pennon: A code for convex nonlinear and semidefinite programming.
\newblock {\em Optimization methods and software}, 18\penalty0 (3):\penalty0
  317--333, 2003.

\bibitem{koiran1994computability}
Pascal Koiran, Michel Cosnard, and Max Garzon.
\newblock Computability with low-dimensional dynamical systems.
\newblock {\em Theor. Comput. Sci.}, 132\penalty0 (1):\penalty0 113--128, 1994.

\bibitem{lanckriet2009convergence}
Gert~R Lanckriet and Bharath~K Sriperumbudur.
\newblock On the convergence of the concave-convex procedure.
\newblock In {\em Advances in neural information processing systems}, pages
  1759--1767, 2009.

\bibitem{Liu2015536}
Hongcheng Liu, Ke~Han, Vikash Gayah, Terry Friesz, and Tao Yao.
\newblock Data-driven linear decision rule approach for distributionally robust
  optimization of on-line signal control.
\newblock {\em Transportation Research Procedia}, 7:\penalty0 536 -- 555, 2015.
\newblock \showISSN{2352-1465}.
\newblock 21st International Symposium on Transportation and Traffic Theory
  Kobe, Japan, 5-7 August, 2015.

\bibitem{marecek2015signaling}
Jakub Mare{\v{c}}ek, Robert Shorten, and Jia~Yuan Yu.
\newblock Signaling and obfuscation for congestion control.
\newblock {\em Int. J. Control}, 88\penalty0 (10):\penalty0 2086--2096, 2015.

\bibitem{marecek2016signaling}
Jakub Mare{\v{c}}ek, Robert Shorten, and Jia~Yuan Yu.
\newblock r-extreme signalling for congestion control.
\newblock {\em Int. J. Control}, 89\penalty0 (10):\penalty0 1972--1984, 2016.

\bibitem{Mishra2012}
V.K. Mishra, K.~Natarajan, Hua Tao, and Chung-Piaw Teo.
\newblock Choice prediction with semidefinite optimization when utilities are
  correlated.
\newblock {\em {IEEE} Trans. Automat. Contr.}, 57\penalty0 (10):\penalty0
  2450--2463, Oct 2012.
\newblock \showISSN{0018-9286}.

\bibitem{Nikolova2014}
Evdokia Nikolova and Nicolas~E. {Stier Moses}.
\newblock A mean-risk model for the traffic assignment problem with stochastic
  travel times.
\newblock {\em Oper. Res.}, 62\penalty0 (2):\penalty0 366--382, 2014.

\bibitem{papageorgiou2007its}
M~Papageorgiou, M~Ben-Akiva, Jon Bottom, Piet~HL Bovy, SP~Hoogendoorn, Nick~B
  Hounsell, Apostolos Kotsialos, and M~McDonald.
\newblock Its and traffic management.
\newblock {\em Handbooks in Operations Research and Management Science},
  14:\penalty0 715--774, 2007.

\bibitem{Soyster1974}
A.~L. Soyster.
\newblock Convex programming with set-inclusive constraints and applications to
  inexact linear programming.
\newblock {\em Oper. Res.}, 21\penalty0 (5):\penalty0 1154--1157, 1973.

\bibitem{Stingl2009}
M.~Stingl, M.~Ko{\v{c}}vara, and G.~Leugering.
\newblock A new non-linear semidefinite programming algorithm with an
  application to multidisciplinary free material optimization.
\newblock In Karl Kunisch, J{\"{u}}rgen Sprekels, G{\"{u}}nter Leugering, and
  Fredi Tr{\"{o}}ltzsch, editors, {\em Optimal Control of Coupled Systems of
  Partial Differential Equations}, volume 158 of {\em International Series of
  Numerical Mathematics}, pages 275--295. Birkh{\"{a}}user Basel, 2009.
\newblock \showISBN{978-3-7643-8922-2}.

\bibitem{toledo2004calibration}
Tomer Toledo, Moshe~E Ben-Akiva, Deepak Darda, Mithilesh Jha, and Haris~N
  Koutsopoulos.
\newblock Calibration of microscopic traffic simulation models with aggregate
  data.
\newblock {\em Transport. Res. Rec.}, 1876\penalty0 (1):\penalty0 10--19, 2004.

\bibitem{vaze2009calibration}
Vikrant Vaze, Constantinos Antoniou, Yang Wen, and Moshe Ben-Akiva.
\newblock Calibration of dynamic traffic assignment models with point-to-point
  traffic surveillance.
\newblock {\em Transport. Res. Rec.}, 2090\penalty0 (1):\penalty0 1--9, 2009.

\bibitem{yen2012convergence}
Ian E.~H. Yen, Nanyun Peng, Po-Wei Wang, and Shou-De Lin.
\newblock On convergence rate of concave-convex procedure.
\newblock In {\em 5th NIPS Workshop on Optimization for Machine Learning},
  2012.

\end{thebibliography}

\end{document}